\documentclass[12pt,a4paper]{article}

\usepackage{amssymb, amsmath, amsthm}

\usepackage[cm]{fullpage}
\usepackage[english]{babel}
\usepackage[cp1250]{inputenc}
\usepackage[T1]{fontenc}
\usepackage[pdftex]{graphicx}
\usepackage{booktabs}

\usepackage{hyperref}

\graphicspath{{./Kody/.vscode}}

\newtheorem{thm}{Theorem}
\newtheorem{lem}{Lemma}

\newcommand{\extrp}{{\overline{U}^{n-\theta}}}

\title{Linear Galerkin-Legendre spectral scheme for a degenerate nonlinear and nonlocal parabolic equation arising in climatology\footnote{This is the accepted version of the manuscript published in \textit{Applied Numerical Mathematics} \textbf{179} (2022), 105-124 with DOI: \url{https://doi.org/10.1016/j.apnum.2022.04.016}}}
\author{\L ukasz P\l ociniczak\thanks{Faculty of Pure and Applied Mathematics, Wroc{\l}aw University of Science and Technology, Wyb. Wyspia{\'n}skiego 27, 50-370 Wroc{\l}aw, Poland}$\;^,$\footnote{Email: lukasz.plociniczak@pwr.edu.pl}}
\date{}

\begin{document}
\maketitle

\begin{abstract}
A special place in climatology is taken by the so-called conceptual climate models. These relatively simple sets of differential equations can successfully describe single mechanisms of climate. We focus on one family of such models based on the global energy balance. This gives rise to a degenerate nonlocal parabolic nonlinear partial differential equation for the zonally averaged temperature. We construct a fully discrete numerical method that has an optimal spectral accuracy in space and second order in time. Our scheme is based on the Galerkin formulation of the Legendre basis expansion, which is particularly convenient for this setting. By using extrapolation, the numerical scheme is linear even though the original equation is nonlinear. We also test our theoretical results during various numerical simulations that support the aforementioned accuracy of the scheme. \\


\end{abstract}

\section{Introduction}
In climate dynamics, one distinguishes between various models according to their complexity and the number of physically resolved phenomena. Overall, there is a hierarchy of models in which on the one end we have General Circulation Models (GCMs), and on the other - Conceptual Climate Models \cite{Cru12, Ran07, Sal02, Sto11}. In between, there are models of Intermediate Complexity that frequently utilize coarser grids or a larger number of parametrizations of various phenomena than GCMs \cite{Cau02}. General Circulation Models are the most advanced descriptions of the global flow of the planetary atmosphere or ocean along with many physical quantities such as pressure, temperature, and water vapour to name only a few. Because they are systems of nonlinear partial differential equations to be solved for long times on a sphere, they usually require supercomputers to analyse and resolve. Thanks to that, they are used to gain quantitative insight about the various interactions between different constituents of the Earth system. Moreover, they are crucial in making predictions of different scenarios concerning, for instance, greenhouse gas emissions and their impact on the mean temperature.

On the other side of the spectrum are Conceptual Climate models that usually are low-order dynamical systems. Their role is to describe several mechanisms of climate along with their nonlinear interdependencies rather than simulate the whole Earth system. They are very useful to investigate the steady states of the climate and their bifurcations which might be difficult to resolve in the full GCMs. For example, conceptual climate models may only focus on energy balance to describe the global or zonally (longitudinally) averaged temperature and steady states of the climate that arise from it. The seminal works concerning this concept have been done by Budyko \cite{Bud69}, Sellers \cite{Sel69} for zero-dimensional case (global mean), and North \cite{Nor75, Nor81} for meridional evolution of temperature (zonal mean). Overall, in energy balance models, one usually assumes that the temperature distribution $T=T(x,t)$, where $x$ denotes the sine of latitude, is governed by the incoming solar radiation, outgoing infrared radiation, and some mechanism of horizontal (meridional) energy transport. Furthermore, as was first noted in \cite{Bha82} the reflectivity of the surface may depend on the history of the temperature. This leads to a nonlinear and nonlocal parabolic equation that is the main subject of our investigations
\begin{equation}
\label{eqn:MainEqTemp}
\begin{cases}
	T_t = (d(T)(1-x^2)T_x)_x + g(x, t, T, JT),	& x\in(0,1), \quad t>0, \\
	T_x(0,t) = 0, \quad T_x(1, t) < \infty, & \\
	T(x,s) = T_0(x,s), \quad -\tau \leq s \leq 0, &
\end{cases}
\end{equation}
where the memory operator with a kernel $K$ is defined as
\begin{equation}
\label{eqn:NonlocalOperator}
J T(x,t) = \int_0^\tau K(s) T(x, t-s) ds.
\end{equation}
Here, $d=d(T)$ is the possibly nonlinear diffusivity while $g$ is the nonlinear source term. Notice that the nonlocality enters the equation through the latter and the second-order differential operator is degenerate for $x=1$. A derivation of the above model in the framework of climate dynamics is given in the next section. In early works, the model was analysed from the climatological point of view in \cite{Nor75a} where the natural efficiency of Legendre orthogonal polynomials was noticed. The linear case was solved exactly for several first modes while the nonlinear one, with the diffusivity proportional to the flux as proposed in \cite{Sto73}, and iteratively in \cite{Lin78}. In this work the authors noted that nonlinear effects have a significant impact on the sensitivity of the model to variations in the solar constant. Moreover, due to a nonlinear source, which is a consequence of the ice-albedo feedback discussed below, the equation can have several steady states representing different climates. Their stability, bifurcations, and sensitivity to parameter perturbations are of high importance in climatology. The simple energy balance model has been generalized in several ways by adding additional degrees of freedom. For the globally averaged model, one can mention adding a mass balance to account for ice sheet variations \cite{Kal97, Ghi81, Plo20, Plo20a} or the amount of CO$_2$ in the atmosphere \cite{Fow13,Fow15}. This is especially relevant for understanding the oscillations of ice ages and their rhythmicity. Further information concerning this topic can be found in \cite{Maa90, Cru12, Dit18, Nym19, Qui19, De13}. A modern review of decades of research on energy balance models can be found in a readable book by North and Kim \cite{Nor17}.

There is also a broad literature concerning the mathematical aspects of the above problem. Its existence, uniqueness, and regularity of solutions was investigated by Hetzer in \cite{Het96}. Moreover, an essential problem for applications - parameter estimation - was analysed in \cite{Roq14, Can20}. As was noted above, steady states of the problem can bifurcate what can lead to hysteresis. Mathematical analysis of this phenomenon was given in \cite{Dia06}. Some other interesting results concerning the local version of the model were obtained by Diaz \cite{Dia97} where the general mathematical theory of energy balance models has been developed for a possibly degenerate nonlinear diffusivity. Furthermore, a generalization of the domain from the 2-sphere into a Riemannian manifold without the boundary was analysed in \cite{Ber09} where also a numerical treatment was conducted. Additionally, since in some parametrizations the source can exhibit a discontinuity in the temperature, a free boundary can arise. We give a climatological background of this phenomenon in the next section and the reader can consult \cite{Dia14} for a thorough treatment. Lastly, we also would like to mention broad studies conducted from the dynamical systems point of view in which the horizontal transport is modelled by a relaxation term for the globally averaged temperature \cite{McG12, Wal14, Wid13}.

There are several accounts of the numerical treatment of energy balance models. The early simulations were based on spectral Legendre decomposition in space and the first-order implicit Euler scheme in time \cite{Nor79}. As was also noted in \cite{Nor75} the model is amenable for such a treatment due to the rapid convergence of the orthogonal series. Recently, several papers concerning the finite element method were published. In \cite{Ber09, Ber08} the situation set on a two-dimensional Riemannian manifold was solved in the case of nonlinear diffusivity modelled by a p-Laplacian. On the other hand, in \cite{Hid15} a finite volume WENO method has been applied to solve the problem where the surface of the Earth is split into land and ocean fractions. Similar analysis but with emphasis on equilibrium solutions was given in \cite{Hid20}.

In this paper, we design and analyse a two-step spectral method based on the Galerkin-Legendre approximation that is weighted in time. In particular, our method contains the second order in time Crank-Nicolson scheme. The main motivation of this paper is to present a rigorous convergence treatment of the early ideas of energy balance simulations \cite{Nor79}. Spectral methods are known for their exponential accuracy provided the regularity of the solution. This can significantly reduce the computation time since only a small number of terms in the expansion needs to be calculated. This is especially relevant for dealing with nonlocal operators in time that require all information about the history in each time-stepping iteration. Our approach is based on similar estimates obtained in the finite element setting presented in the classical works of Douglas and Dupont \cite{Dou73}. However, as is seen from (\ref{eqn:MainEqTemp}) our PDE becomes degenerate at $x=1$ which causes several difficulties. We overcome them by introducing a weighted $L^2$ space in which the solution is sought. We also utilize the extrapolation of coefficients in the same way as was done in \cite{Lus79} to obtain a linear method. Galerkin finite element method has also been applied to parabolic equations with nonlocal terms in \cite{Can90}. On the other hand, Galerkin spectral method has recently been applied to nonlinear parabolic problems in \cite{Liu16}. The reader can find state-of-the-art surveys of all variants of spectral methods in \cite{Can07,She11}.

The paper is structured as follows. In the next section, we present a derivation of our main model (\ref{eqn:MainEqTemp}) in the climatological setting. In Section 3, we deal with a semidiscrete scheme where we discretize the spatial variable leaving a continuous dependence on time. The method attains the optimal spectral accuracy. The convergence proofs are given. Furthermore, Section 4 concerns the fully discrete method where a two-step weighted scheme with extrapolated coefficients is constructed and analysed with respect to the convergence. Thanks to the extrapolation, our method is linear despite the fact that the original problem may be highly nonlinear due to the diffusivity. In Section 5, we provide implementation guidelines, solutions that improve calculation performance, and simulations that verify previously proved estimates. We end the paper with a conclusion with some prospects for future work.

In what follows, the generic letter $C$ will be used frequently to denote any positive constant that can depend on the solution and its derivatives but not on the discretization parameters. Moreover, as in the usual practice, the value of $C$ can change even in the same chain of  inequalities. This will not cause any confusion or lack of rigour. 

\section{Background on climate dynamics}
To set the stage for our subsequent reasoning, we review the derivation of the main equations in the setting of climate dynamics. Some more elaborate exposition can be found in \cite{Fow11,Plo20a}. 

We consider a simple conservation of energy written in terms of the zonally and vertically averaged temperature that is symmetric with respect to the equator. Let $T = T(x, t)$ be the mean temperature at time $t$ and latitude $\theta$ where $x = \sin\theta$. We prefer to use $x$ instead of $\theta$ since then, the spherical Laplacian simplifies substantially. The basic model is the following
\begin{equation*}
c T_t = R_i - R_o + H,
\end{equation*}   
where $c$ is the heat capacity of Earth, $R_i$ is the incoming solar radiation that reaches the surface (insolation), $R_o$ is the outgoing infrared radiation, and $H$ represents the horizontal transport. 

We can further parametrize these various constituents. As the amount of heat $Q$ reaches Earth its fraction $\alpha$, called albedo, is reflected by the surface. Since, for example, snow ($\alpha = 0.80$) reflects much more light than the ocean ($\alpha = 0.06$), we have to distinguish between different types of surface and its albedo. This leads to the so-called ice-albedo feedback that is one of the most important mechanisms regulating the climate. Since ice has a small albedo, it causes more radiation to be reflected, lowering the surface temperature. This, in turn, produces favourable conditions for the formation of ice caps. This phenomenon can be parametrized by letting $\alpha$ to depend on $x$ and $T$. In the literature one can find many of these functional relations. For example, Budyko proposed that the albedo has two values: one for ice and one for ice-free surface. The boundary between these, called the ice line, depends on the temperature
\begin{equation}
\label{eqn:AlbedoFreeBoundary}
\alpha(x,t) = 
\begin{cases}
	\alpha_1, & T(x,t) \leq T_i, \\
	\alpha_2, & T(x,t) > T_i, 
\end{cases}
\end{equation}
where $T_i$ is usually taken as $-10^\circ $C. That is, ice starts to form when the mean temperature is below a threshold. This parametrization leads to a free-boundary problem since one has to determine the ice line position $x_i(t)$ satisfying $T(x_i(t),t) = T_i$. This approach has been analysed, for example in \cite{Dia20}. Other forms of albedo have also been proposed. One of the most common ones assumes continuous dependence on both latitude and the temperature. For example, Sellers suggested a piecewise linear relationship
\begin{equation*}
\alpha(T) = 
\begin{cases}
	\alpha_1, & T\leq T_1, \\
	\alpha_1 + (\alpha_2-\alpha_1)\dfrac{T-T_1}{T_2-T_1}, & T_1 < T \leq T_2, \\
	\alpha_2, & T > T_1, 
\end{cases}
\end{equation*} 
for some choices of $T_{1,2}$. In general, albedo is frequently represented as a bounded, monotone increasing function of the temperature and, possibly, a low-order polynomial in space \cite{Kal97, Bha82, Fow13, Plo20a}. We will make such an assumption in the sequel. Moreover, as was suggested in \cite{Bha82} (but also see \cite{Qui18}), the present albedo of the ice-covered ground is determined not only by the actual temperature, but rather by its past values. Whence, we take $\alpha$ to be a function of the nonlocal (memory) operator acting on $T$, that is
\begin{equation*}
R_i = Q S(x, t) (1-\alpha(x,T, JT)),
\end{equation*}
where $S=S(x,t)$ is the distribution of insolation over latitude which takes into account an uneven illumination of the spherical Earth. To a good approximation, one can take $S(x) \approx S_0 + S_1 L_2(x)$, where $L_2$ is the second Legendre polynomial \cite{Nor17}. 

Furthermore, the absorbed heat is isotropically reradiated into space. This can be parametrized by the Stefan-Boltzmann law (as was done by Sellers)
\begin{equation*}
R_o = \sigma T^4,
\end{equation*}
or its linearization $R_o = A + BT$ (as was done by Budyko). Since the change in the temperature is relatively small, this simplification is usually justified. 

Finally, we have the horizontal transport term that arises due to the uneven temperature distribution along the latitude: heat moves from the equator to the polar regions. To the level of complexity that we want to achieve, we assume that the transport can be modelled by a diffusion term \cite{Nor75}
\begin{equation*}
H = \nabla\cdot \left(d(u) \nabla T\right) = \left(d(u) (1-x^2)T_x\right)_x.
\end{equation*}
The other approach, by Budyko, is based on using a relaxation term $\propto (T-\overline{T})$ where $\overline{T}$ is the global mean temperature. The boundary conditions are taken to model the vanishing energy flux on both the equator $x=0$ and poles $x=1$, that is we impose $-D(u)(1-x^2)T_x = 0$ there. As was found in \cite{Nor75a, Nor75} linear diffusion, i.e., $D(u) = $const., leads to an accurate and robust model and thus, it is an important case to consider. However, as was suggested in \cite{Sto73} and further verified in \cite{Lin78}, having the diffusivity being a function of the gradient has a significant effect on the stability of polar ice caps. This parametrization leads to the p-Laplacian operator which, in this case, was analysed in \cite{Dia97}. Putting all obtained terms and renaming the source term, we arrive at (\ref{eqn:MainEqTemp}). 

We note that the heat capacity $c$ is assumed to be constant, however, in \cite{Roq14} it was suggested that it may also be a function of the past values of temperature and latitude. Investigating a model of this type along with gradient dependent diffusivity will be the subject of our future work. 

\section{Spectral discretization with respect to space}

\subsection{Definitions and assumptions}
We begin with some preparations. By $L^2(0,1)$ we denote the usual Hilbert space of square-integrable functions with a norm $\left\|\cdot \right\|$. Similarly, $H^m(0,1)$ where $m\geq 1$ is the Sobolev space of $m-$th weakly differentiable functions with the norm denoted by $\left\|\cdot\right\|_m$. As for the scalar product, we will only use $L^2$ one and write $(\cdot, \cdot)$. We also introduce an intermediate space where the solution to (\ref{eqn:MainEqTemp}) lives 
\begin{equation*}
V = \left\{v \in H^1(0,1): \, \sqrt{1-x^2}\,v_x \in L^2(0,1)\right\},
\end{equation*}
with the norm
\begin{equation}
\label{eqn:VNorm}
\left\| v \right\|_V = \int_0^1 (1-x^2) v_x^2 \, dx + \int_0^1 v^2 dx.
\end{equation}
Due to the degeneracy of the equation (\ref{eqn:MainEqTemp}) at $x=1$, the above weighted $L^2$ space is a natural choice. In addition, because of that reason, the corresponding quadratic form needed for the definition of a weak solution is not coercive (it is only weakly coercive). One of the standard ways of dealing with that problem is to introduce a transformation
\begin{equation*}
u(x,t) = e^{-t} T(x,t),
\end{equation*}
which leads to
\begin{equation}
\label{eqn:MainEq}
\begin{cases}
	u_t + u = (D(u)(1-x^2)u_x)_x + f(x, t, u, Ju),	& x\in(0,1), \quad t\in (0, t_0), \\
	u_x(0,t) = 0, \quad u_x(1, t) < \infty, & \\
	u(x,s) = \psi(x,s), \quad -\tau \leq s \leq 0, &
\end{cases}
\end{equation}
where $f(x, t, u, Ju) = e^{-t} g(x, t, e^{t} u, J (e^t u))$ and $\psi(x,s) = e^{-t} T_0(x,s)$. Notice that we should have written $D(u) = d(e^t u)$, i.e. indicating the explicit dependence on time. However, since $t\in[0,t_0]$ we have $e^t$ bounded. Therefore, omitting it from the diffusivity will not produce any quantitative effects in the proofs below. We thus commit this slight abuse of notation and abstain from writing explicit time dependence. The reader will see every reasoning below can be repeated for time-dependent cases with essentially no changes. 

Now, by multiplying the above by $v\in V$ and integrating by parts from $x = 0$ to $x=1$ we obtain a weak formulation which is the basis for the Galerkin method
\begin{equation}
\label{eqn:WeakForm}
\begin{cases}
	(u_t, v) + a(D(u); u, v) = (f(t, u, Ju), v), \quad v \in V, \\
	u(s) = \psi(s), \quad -\tau \leq s \leq 0,
\end{cases}
\end{equation}
where the quadratic form $a$ linear in the second and third argument is defined by
\begin{equation}
\label{eqn:aForm}
a(D(w); u, v) = \int_0^1 D(w) (1-x^2) u_x v_x dx + \int_0^1 u v \, dx = (D(u)(1-x^2) u_x, v_x) + (u, v).
\end{equation}
We will also write $a(u,v) := a(1; u, v)$ which implies that $a(u,u) = \left\|u\right\|_V^2$. Moreover, from now on, if it does not pose any threat to unambiguity, we will suppress writing the independent variables.

Concerning the assumptions imposed on various parameters, we make the natural choices that are also required for existence and uniqueness (see \cite{Dia97,Roq14,Het96}). Specifically, we assume that the diffusivity $D=D(u)$ and the source $f=f(x,t,u,w)$ are smooth with
\begin{equation}
\label{eqn:Assumptions}
0 < D_- \leq D(u) \leq D_+ < \infty, \quad |D_u| + |f_u| + |f_w| \leq C.
\end{equation}
Moreover, we take the kernel of the nonlocal operator (\ref{eqn:NonlocalOperator}) to have a minimal regularity for well-posedness
\begin{equation*}
K \in L^1(0,t_0).
\end{equation*}
In several places, we will further assume that the solution of (\ref{eqn:MainEq}) is sufficiently smooth, which, in turn, requires more regularity on $D$ and $f$. We also would like to note that considering a discontinuous source case is one of the subjects of our future work. 

\subsection{Numerical scheme}
In the spatial discretization, we use the Galerkin scheme for which the test and trial functions (see \cite{Can07}) belong to the following finite-dimensional space
\begin{equation*}
V_N = \left\{v \in \mathbb{P}_N(0,1): \, v_x(0) = 0, \; v_x(1) < \infty\right\},
\end{equation*}
where $\mathbb{P}_N$ is the polynomial space of degree $N$. As for the orthonormal basis $\left\{\phi_i\right\}_{i=0}^N$ for $V_N$ we choose 
\begin{equation}
\label{eqn:Basis}
\phi_i = \sqrt{4n+1} L_{2i}(x), \quad i=0,1,2,...,
\end{equation} 
where $L_{2i}$ is the Legendre polynomial of degree $2i$. We obviously have $(\phi_i,\phi_j) = \delta_{ij}$. Moreover, this choice is particularly convenient for linear diffusion, i.e., $D(u) = const.$, because it constitutes the eigenbasis for the second-order operator
\begin{equation}
\label{eqn:Eigenfunctions}
L\phi_i := -\left((1-x^2) (\phi_{i,x})\right)_x = \lambda_i \phi_i, \quad \lambda_i = 2i(2i+1).
\end{equation}
This automatically diagonalizes the stiffness matrix and facilitates computations in the important linear case or makes the matrix sparse for weakly nonlinear diffusion. 

By weighting (\ref{eqn:WeakForm}) with respect to $V_N$ we formulate the Legendre-Galerkin numerical scheme. We thus look for $u_N(t)\in V_N$ that for all $t\in[0,t_0]$ satisfies
\begin{equation}
\label{eqn:GalerkinWeakForm}
\begin{cases}
	(u_{N,t}, v) + a(D(u_N); u_N, v) = (f(t, u_N, Ju_N), v), \quad v \in V_N, \\
	u_N(s) = \psi_N(s), \quad -\tau \leq s \leq 0,	
\end{cases}
\end{equation}
where $\psi_N$ is the appropriate approximation to the initial condition. Note also that we have refrained from writing the $x$ variable as arguments since the inner product is taken with respect to it. This slight abuse of notation should not cause any misunderstandings. Henceforth, we will use the orthogonal $L^2$ Legendre projection 
\begin{equation}
\label{eqn:OrthogonalProjection}
\psi_N(s) = P_N \psi(s) := \sum_{i=0}^N (\psi(s),\phi_i) \phi_i, \quad -\tau \leq s \leq 0,	
\end{equation}
which has the following spectral accuracy (for a detailed discussion see \cite{Can07}, formulas (5.4.12) and (5.4.17))
\begin{equation}
\label{eqn:OrthogonalProjectionAccuracy}
\begin{split}
	&\left\|u - P_N u \right\| \leq C N^{-m} \left\|u\right\|_m, \quad \left\|u - P_N u \right\|_l \leq C N^{2l - 1/2 - m} \left\|u\right\|_m, \\
	&\left\|u - P_N u \right\|_\infty \leq C N^{\frac{1}{2}-m} V(u^{(m)}_x), 
\end{split}
\end{equation}
where $V(\cdot)$ denoted the total variation and $l\geq 1$. The error of the approximation with $P_N$ in $V$-norm is better than that in the Sobolev space.
\begin{lem}
\label{lem:VNormPEstimate}
Let $u(t)\in H^{2m}(0,1)$ for $m\geq 1$ and each $t\in[0,t_0]$ with $u$ and $u_t$ bounded. Then, for sufficiently large $N$ we have
\begin{equation}
	\label{eqn:VNormPEstimate}
	\left\|u-P_N u\right\|_V \leq C N^{1-2m} \left\|L^m u\right\| \leq C N^{1-2m} \left\|u\right\|_{2m},
\end{equation}
where $L$ is defined in (\ref{eqn:Eigenfunctions}). 
\end{lem}
\begin{proof}
We start by writing
\begin{equation*}
	((1-x^2) (u-P_N u)_x, (u-P_N u)_x) = \int_0^1 (u-P_N u) L(u-P_N u) dx,
\end{equation*}
which follows from the integration by parts. Since 
\begin{equation*}
	u - P_N u = \sum_{i=N+1}^\infty (u, \phi_i) \phi_i,
\end{equation*}
we obtain
\begin{equation*}
	((1-x^2) (u-P_N u)_x, (u-P_N u)_x) = \sum_{i,j=N+1}^\infty (u, \phi_i) (u, \phi_j) \int_0^1 \phi_i L \phi_j dx.
\end{equation*}
Moreover, since $\phi_i$ is the eigenfunction of $L$ (see (\ref{eqn:Eigenfunctions})) we have
\begin{equation*}
	((1-x^2) (u-P_N u)_x, (u-P_N u)_x) = \sum_{i=N+1}^\infty \lambda_i |(u, \phi_i)|^2.
\end{equation*}
However, we can also write $\phi_i = \lambda_i^{-1} L\phi_i$ which after iteration implies that
\begin{equation*}
	((1-x^2) (u-P_N u)_x, (u-P_N u)_x) = \sum_{i=N+1}^\infty \lambda^{1-2m}_i |(u, L^m\phi_i)|^2 \leq \lambda_{N+1}^{1-2m} \left\|L^m u \right\|^2,
\end{equation*}
where the last inequality follows from Plancherel's identity. Further, since $\lambda_{N+1} = (2N+2)(2N+3)$ we can write
\begin{equation*}
	((1-x^2) (u-P_N u)_x, (u-P_N u)_x) \leq C N^{2-4m} \left\|L^m u \right\|^2.
\end{equation*} 
Now, by the $L^2$ error estimates for $P_N$ given in (\ref{eqn:OrthogonalProjectionAccuracy}) we have
\begin{equation*}
	((1-x^2) (u-P_N u)_x, (u-P_N u)_x) + (u-P_N u, u-P_N u) \leq C \left( N^{2-4m} \left\|L^m u \right\|^2 + N^{-4m} \left\|u\right\|_{2m}^2\right),
\end{equation*} 
which for sufficiently large $N$ implies the sought error estimate in the $V$-norm.
\end{proof}

We note that the above is not the only optimal choice of a projection that we can use. For the main equation (\ref{eqn:GalerkinWeakForm}) it is much more convenient to use the following elliptic (or Ritz) orthogonal projection
\begin{equation}
\label{eqn:RitzProjection}
a(D(u); R_N u - u, v) = 0, \quad v \in V_N,
\end{equation}
which utility in obtaining an optimal order of convergence was discovered in the early days of mathematical finite element analysis \cite{Tho07}. 

Before we proceed the convergence result for (\ref{eqn:GalerkinWeakForm}) we state several auxiliary lemmas concerning Ritz projection (\ref{eqn:RitzProjection}). First, we show that it has the optimal order of approximation both in $V$ and in $L^2(0,1)$. 
\begin{lem}
\label{lem:RitzProjection}
Let $u(t)\in H^{2m}(0,1)$ for $m\geq 1$ and each $t\in[0,t_0]$ with $u$ and $u_t$ bounded. Then, for sufficiently large $N$ we have the following estimate in $V$ and $L^2$
\begin{equation}
	\label{eqn:RitzEstimate}
	\left\|u - R_N u\right\|+N^{-1}\left\|u - R_N u\right\|_V \leq C N^{-2m}\left\|u\right\|_{2m}.
\end{equation}
Moreover, if additionally $u_t\in H^{2m}(0,1)$ then the time derivatives of the errors satisfy
\begin{equation}
	\label{eqn:RitzDerivativeEstimate}
	\left\|(u - R_N u)_t\right\|+N^{-1}\left\|(u - R_N u)_t\right\|_V \leq C N^{-2m}\left(\left\|u\right\|_{2m}+\left\|u_t\right\|_{2m}\right).
\end{equation}
\end{lem}
\begin{proof}
By (\ref{eqn:aForm}) and since $R_N u \in V_N$ we have
\begin{equation*}
	D_- \left\|u - R_N u\right\|_V^2 \leq a(D(u); u-R_N u, u-R_N u) = a(D(u); u-R_N u, u-v) + a(D(u); u-R_N u, v - R_N u),
\end{equation*}
where $v\in V_N$. By the orthogonality of the Ritz projection (\ref{eqn:RitzProjection}) the last term above vanishes leaving
\begin{equation*}
	D_- \left\|u - R_N u\right\|_V^2 \leq D_+ a(u - R_N u, u - v), \quad v \in V_N. 
\end{equation*}
By choosing $v = P_N u$ and using Cauchy-Schwarz inequality, we immediately have
\begin{equation*}
	\left\|u - R_N u\right\|_V^2 \leq \frac{D_+}{D_-} \left\|u - R_N u\right\|_V \left\|u - P_N u\right\|_V.
\end{equation*}
Furthermore, the application of Lemma \ref{lem:VNormPEstimate} leads to
\begin{equation}
	\label{eqn:RitzEstimateV}
	\left\|u - R_N u\right\|_V \leq \frac{D_+}{D_-} \left\|u - P_N u\right\|_V \leq C N^{1-2m} \left\|u\right\|_m,
\end{equation}
which is the first assertion. 

To show the $L^2$ estimate of the error, we follow the duality argument (see \cite{Tho07}). For a fixed $u \in H^{2m}(0,1)$ let $w$ be the solution of
\begin{equation}
	\label{eqn:Duality}
	a(D(u); w, v) = (u-R_N u, v), \quad v \in V.
\end{equation}
By putting $v = w$ we immediately obtain the stability estimate $\left\|w\right\|_V \leq C \left\|u-R_N u\right\|$. Using the definition of $a$ as in (\ref{eqn:aForm}) and reintegrating by parts, we can obtain
\begin{equation}
	\label{eqn:LEstimate}
	\left\|Lw \right\| \leq C \left\|u-R_N u\right\| + \left\|w\right\| \leq C \left\|u-R_N u\right\| + \left\|w\right\|_V \leq C \left\|u-R_N u\right\|,
\end{equation}
where the assumption of bounded $D$ has been used in the first inequality while the stability estimate in the last. Now, by choosing $v = u - R_N u$ in (\ref{eqn:Duality}) and using the definition of the Ritz projection (\ref{eqn:RitzProjection}) we have
\begin{equation*}
	\begin{split}
		\left\|u-R_N u\right\|^2 
		&= (u-R_N u, u-R_N u) = a(D(u); w, u-R_N u) \\
		&= a(D(u); w - P_N w, u - R_N u) + a(D(u); P_N w, u - R_N u) \\
		&= a(D(u); w - P_N w, u - R_N u).
	\end{split}
\end{equation*}
Furthermore, we can use Cauchy-Schwarz inequality along with (\ref{eqn:RitzEstimateV}) and (\ref{eqn:VNormPEstimate}) to infer that
\begin{equation*}
	\left\|u-R_N u\right\|^2 \leq D_+ \left\|w - P_N w\right\|_V \left\|u - R_N u\right\|_V \leq C N^{1-2} \left\|Lw\right\| N^{1-2m} \left\|u\right\|_{2m} .
\end{equation*}
The proof of (\ref{eqn:RitzEstimate}) is finished after utilizing (\ref{eqn:LEstimate}).

The reasoning used in showing the time differentiated version (\ref{eqn:RitzDerivativeEstimate}) is similar and based on a derivative of the Ritz projection definition (\ref{eqn:RitzProjection})
\begin{equation}
	\label{eqn:RitzProjectionDerivative}
	a(D(u); (R_N u - u)_t, v) + a(D(u)_t; R_N u - u, v) = 0, \quad v \in V_N.
\end{equation}
From here it follows that
\begin{equation*}
	\begin{split}
		D_- \left\|(R_N u - u)_t\right\|_V^2 
		&\leq a(D(u); (R_N u - u)_t, (R_N u - u)_t) \\
		&= a(D(u); (R_N u)_t - v, (R_N u - u)_t) + a(D(u); v - u_t, (R_N u - u)_t) \\
		&= a(D(u)_t; v-(R_N u)_t, R_N u - u) + a(D(u); v - u_t, (R_N u - u)_t).
	\end{split}
\end{equation*}
Now, Cauchy-Schwarz inequality yields
\begin{equation*}
	D_-\left\|(R_N u - u)_t\right\|_V^2 \leq C \left(\left\|v-(R_N u)_t\right\|_V \left\|R_N u - u\right\|_V + \left\|v-u_t\right\|_V\left\|(R_N u - u)_t\right\|_V\right).
\end{equation*}
We can make the norm of differences sufficiently small by an orthogonal approximation, i.e., by choosing $v = P_N u_t$. Whence,
\begin{equation*}
	D_-\left\|(R_N u - u)_t\right\|_V^2 \leq C \left(\left\|(P_N u - R_N u)_t\right\|_V \left\|R_N u - u\right\|_V + \left\|(P_N u-u)_t\right\|_V\left\|(R_N u - u)_t\right\|_V\right),
\end{equation*}
and by another estimate
\begin{equation*}
	\left\|(P_N u - R_N u)_t\right\|_V \leq \left\|(P_N u - u)_t\right\|_V + \left\|(u - R_N u)_t\right\|_V,
\end{equation*}
we can write
\begin{equation*}
	\begin{split}
		D_-\left\|(R_N u - u)_t\right\|_V^2 
		&\leq C \left(\left\|(P_N u - u)_t\right\|_V \left\|R_N u - u\right\|_V + \left\|(u - R_N u)_t\right\|_V \left\|R_N u - u\right\|_V \right. \\ 
		&\left.+ \left\|(P_N u-u)_t\right\|_V\left\|(u - R_N u)_t\right\|_V\right).
	\end{split}
\end{equation*}
We can now use the $\epsilon$-Cauchy inequality, that is
\begin{equation*}
	a b \leq \frac{\epsilon}{2}a^2 + \frac{1}{2\epsilon} b^2, \quad a,b\in\mathbb{R}, \quad \epsilon > 0,
\end{equation*}
to transform our estimate to
\begin{equation*}
	\begin{split}
		D_-\left\|(R_N u - u)_t\right\|_V^2 \leq \frac{D_-}{2} \left\|(R_N u - u)_t\right\|_V^2 + C \left(\left\|(P_N u - u)_t\right\|_V + \left\|u - R_N u\right\|_V\right)^2,
	\end{split}
\end{equation*}
that is
\begin{equation*}
	\left\|(R_N u - u)_t\right\|_V \leq C \left(\left\|(P_N u - u)_t\right\|_V + \left\|u - R_N u\right\|_V\right).
\end{equation*}
Hence, the $V$-norm part of (\ref{eqn:RitzDerivativeEstimate}) follows from (\ref{eqn:VNormPEstimate}) and (\ref{eqn:RitzEstimate}). 

The $L^2$ estimate again follows from the duality argument. Let $w$ be as before in (\ref{eqn:Duality}) but $(u - R_N u)_t$ on the right-hand side. Reasoning similarly as before we have with $v\in V$
\begin{equation*}
	\begin{split}
		\left\|(u - R_N u)_t\right\|^2 
		&= a(D(u); w, (u - R_N u)_t) \\
		&= a(D(u); w-v, (u - R_N u)_t) + a(D(u); v, (u - R_N u)_t) \\
		&= a(D(u); w-v, (u - R_N u)_t) - a(D(u)_t; v, u - R_N u) \\
		&= a(D(u); w-v, (u - R_N u)_t) + a(D(u)_t; w-v, u - R_N u) - a(D(u)_t; w, u - R_N u).
	\end{split}
\end{equation*}
where in the third equality we have moved the derivative according to (\ref{eqn:RitzProjectionDerivative}) while in the last we have introduced $w$. Thanks to that, with $v = P_N w$ we can bound the above by previously obtained estimates in the $V$-norm
\begin{equation*}
	\begin{split}
		\left\|(u - R_N u)_t\right\|^2 
		&= C \left(\left\|w-P_N w\right\|_V \left(\left\|(u - R_N u)_t\right\|_V+\left\|u - R_N u\right\|_V\right) + |a(w-P_N w, u - R_N u)|\right).
	\end{split}	
\end{equation*}
The first two terms can be tackled exactly in the same way as above, and the third term can be integrated by parts to obtain
\begin{equation*}
	\begin{split}
		|a(w-P_N w, u - R_N u)| 
		&\leq \int_0^1 |Lw| |u - R_N u| dx \leq \left\|Lw \right\| \left\|u - R_Nu\right\| \\
		&\leq C \left\|u - R_Nu \right\| \left\|u - R_Nu\right\|,
	\end{split}
\end{equation*}
by stability estimate for elliptic equation (\ref{eqn:LEstimate}). Finally, combining the two above estimates with (\ref{eqn:RitzEstimate}) yields (\ref{eqn:RitzProjectionDerivative}) and finishes the proof. 
\end{proof}
As we have seen, the elliptic projection has the same order of accuracy as the standard projection $P_N$. Note also that the error in the $V$-norm is bounded by $N^{1-2m}$ for sufficiently regular functions $u$. This norm involves the first derivative and is similar to the standard $H^1$ norm with $\left\|u\right\|_V \leq \left\|u\right\|_1$. This can be compared with the classical result concerning the approximation in $H^1$ which states that in that case the error is proportional to $N^{3/2-2m}$ which is larger than the former estimate. This again shows that choosing $V$ as the trial space is very natural and optimal.

The Ritz projection is also bounded in the maximum norm which is shown below.
\begin{lem}\label{lem:RitzGradient}
Let $R_N$ be defined as in (\ref{eqn:RitzProjection}) and $u\in H^4(0,1)$. Then, for sufficiently large $N$ we have
\begin{equation*}
	\left\|\left(R_N u\right)_x\right\|_\infty \leq  \left\|u_x\right\|_\infty.
\end{equation*} 
\end{lem}   
\begin{proof}
We will use the following polynomial inverse inequalities, one for differentiation (Markov inequality, \cite{Tim14}, p. 218), and the others for summability (see \cite{Can07}, formulas (5.4.3) and (5.4.5))
\begin{equation*}
	\left\|v_x\right\|_\infty \leq C N^2 \left\|v\right\|_\infty \leq C N^3 \left\|v\right\|, \quad v\in \mathbb{P}_N. 
\end{equation*}
Now, we can write
\begin{equation*}
	\left\|(R_N u)_x\right\|_\infty \leq \left\|(R_N u - u)_x\right\|_\infty + \left\|u_x\right\|_\infty,
\end{equation*}
and use the above inverse inequality and the error estimate (\ref{eqn:RitzEstimate})
\begin{equation*}
	\left\|(R_N u)_x\right\|_\infty \leq C N^3\left\|R_N u - u\right\|+ \left\|u_x\right\|_\infty \leq C N^{3-4} \left\|u\right\|_4 + \left\|u_x\right\|_\infty,
\end{equation*} 
from which the conclusion follows for sufficiently large $N$.
\end{proof}

The last auxiliary result is the classical Gr\"onwall-Bellman's lemma that we state without the proof (which can be found in \cite{Lin85}).
\begin{lem}[Gr\"onwall-Bellman]
\label{lem:GB}
Let $F(t)$ and $y(t)$ be continuous, non-decreasing, and nonnegative functions. Then
\begin{equation*}
	y(t) \leq F(t) + C \int_0^t y(s) ds,
\end{equation*}
implies
\begin{equation*}
	y(t) \leq F(t) e^{C t}.
\end{equation*}
\end{lem}

We are ready to prove the main results of this section concerning the semidiscrete numerical scheme. First, we show that the method is stable in time.
\begin{thm}[Stability]
Assume that $f = f(x,t,u,w)$ is bounded with respect to $u$, and $w$, that is
\begin{equation*}
	\left\|f(x,t,u,w)\right\|_\infty \leq g(x,t).
\end{equation*}
where $g=g(x,t)$ is smooth. Then, if $u_N$ is the solution of (\ref{eqn:GalerkinWeakForm}), we have
\begin{equation*}
	\left\|u_N(t) \right\| \leq e^{-D_- (\lambda_1-1) t}\left\|\psi_N(0)\right\| + \int_0^1 e^{-D_- (\lambda_1-1) (t-s)}\left\|g(s)\right\| ds,
\end{equation*}
where $D_-$ is given in (\ref{eqn:Assumptions}) while $\lambda_1$ is defined in (\ref{eqn:Eigenfunctions}). 
\end{thm} 
\begin{proof}
The proof is standard: in (\ref{eqn:GalerkinWeakForm}) choose $v = u_N$, use Cauchy-Schwarz inequality, and the boundedness assumption to obtain
\begin{equation*}
	\frac{1}{2}\frac{d}{dt}\left\|u_N(t)\right\|^2 + a(D(u_N); u_N, u_N) \leq \left\|g(t)\right\| \left\|u_N(t)\right\|.
\end{equation*}
Now, since $a(D(u_N); u_N, u_N) \geq  D_- a(u_N, u_N)$ (see (\ref{eqn:aForm}) and (\ref{eqn:Assumptions})) we can write
\begin{equation*}
	\left\|u_N(t)\right\|\frac{d}{dt}\left\|u_N(t)\right\| + D_- a(u_N, u_N) \leq \left\|g(t)\right\| \left\|u_N(t)\right\|,
\end{equation*}
which after division by $\left\|u_N(t)\right\|$ leads to
\begin{equation*}
	\frac{d}{dt}\left\|u_N(t)\right\| + D_-  \frac{a(u_N, u_N)}{\left\|u_N(t)\right\|^2} \left\|u_N(t)\right\| \leq \left\|g(t)\right\|.
\end{equation*}
The left-hand side is bounded from below by the infimum over the whole space $V$, hence
\begin{equation*}
	\frac{d}{dt}\left\|u_N(t)\right\| + D_- \inf_{v\in V} \frac{a(v, v)}{\left\|v\right\|^2} \left\|u_N(t)\right\| \leq \left\|g(t)\right\|.
\end{equation*}
According to the standard theory of elliptic PDEs, the infimum can be interpreted as
\begin{equation*}
	\inf_{v\in V} \frac{a(v, v)}{\left\|v\right\|^2} = \lambda_1-1,
\end{equation*}
where $\lambda_1-1$ is the smallest eigenvalue of the problem
\begin{equation*}
	-((1-x^2)u_x)_x + u = \lambda u, \quad \lambda \in \mathbb{R}.
\end{equation*}
The value of $\lambda_1$ is given in (\ref{eqn:Eigenfunctions}). Finally, multiplying the inequality by the factor $e^{D_- (\lambda_1-1) t}$ and integrating yields the sought result. 
\end{proof}
Again, we can see the close relationship between the linear diffusion generated by the operator $L$ and the fully nonlinear case. We can now proceed to the convergence proof. 
\begin{thm}[Convergence]
Let $u(t)$ and $u_N(t)$ be solutions of (\ref{eqn:WeakForm}) and (\ref{eqn:GalerkinWeakForm}), respectively belonging to $H^{2m}(0,1)$ with $m\geq 2$ for each $t\in[0,t_0]$. Moreover, assume that $u_x$ and $u_t$ are bounded. Then, for sufficiently large $N$ and for each $t\in [0, t_0]$ we have
\begin{equation*}
	\begin{split}
		&\left\|u(t) - u_N(t)\right\| 
		\leq \left\|R_N(t) - u(t)\right\| \\
		&+C \left(\left\|P_N \psi(0) - \psi(0)\right\| + \left\|R_N \psi(0) - \psi(0)\right\| + \int_0^t\left(\left\|(R_N - u)_t(z)\right\| + \left\|(R_N - u)(z)\right\|\right) dz\right),
	\end{split}
\end{equation*}
in particular
\begin{equation*}
	\left\|u(t) - u_N(t)\right\| \leq C N^{-2m},
\end{equation*}
where the constant $C$ depends on $u$, $u_x$, $u_t$, $D$, $f$, and $K$. 
\end{thm}
\begin{proof}
We start by decomposing the error
\begin{equation*}
	u - u_N = u - R_N u + R_N u - u_N =: r_N + e_N.
\end{equation*}
Therefore, since we have already proved the estimates on $r_N$ in Lemma \ref{lem:RitzProjection} we have to focus only on $e_N$. To this end, we will write the error equation. For $v\in V_N$ we have
\begin{equation*}
	\begin{split}
		(e_{N,t}, v) + a(D(u_N); e_N, v) 
		&= (R_N u_t, v) + a(D(u_N); R_N u, v) - (u_{N,t}, v) - a(D(u_N); u_N, v) \\
		&= (R_N u_t, v) + a(D(u_N); R_N u, v) - (f(t,u_N, J u_N), v),
	\end{split}
\end{equation*}
where we have used (\ref{eqn:GalerkinWeakForm}). Furthermore, by using the main equation (\ref{eqn:WeakForm}) and the definition of the Ritz projection (\ref{eqn:RitzProjection}) we can obtain
\begin{equation*}
	\begin{split}
		(e_{N,t}, v) &+ a(D(u_N); e_N, v) \\
		&= (R_N u_t, v) + a(D(u); R_N u, v) + a(D(u_N)-D(u); R_N u, v) - (f(t,u_N, J u_N), v) \\
		&= (R_N u_t, v) + a(D(u); u, v) + a(D(u_N)-D(u); R_N u, v) - (f(t,u_N, J u_N), v) \\
		&= ((R_N u - u)_t, v) + a(D(u_N)-D(u); R_N u, v) - (f(t,u_N, J u_N)-f(t,u, J u), v).
	\end{split}
\end{equation*}
Therefore, by taking $v = e_N \in V_N$ we are led to the estimate
\begin{equation*}
	\begin{split}
		\frac{1}{2}\frac{d}{dt}\left\|e_N\right\|^2 + D_- \left\|e_N\right\|_V^2 
		&\leq |(r_{N,t}, e_N)| + |a(D(u_N)-D(u); R_N u, e_N)| \\
		&+ |(f(t,u_N, J u_N)-f(t,u, J u), e_N)| =: \rho_1 + \rho_2 + \rho_3.
	\end{split}
\end{equation*}
And from here we are left to estimate the three remainders $\rho_i$. The first one comes from (\ref{eqn:RitzDerivativeEstimate}), while for the second we use (\ref{eqn:Assumptions}) and obtain
\begin{equation*}
	\rho_2 \leq C \int_0^1 |u_N - u| (1-x^2)|(R_N u)_x| |e_{N,x}| dx.
\end{equation*}
Now, the gradient of the Ritz projection is bounded according to Lemma \ref{lem:RitzGradient} and hence
\begin{equation*}
	\rho_2 \leq C \int_0^1 |u_N - u| (1-x^2)|e_{N,x}| dx \leq C \left\|u_N - u \right\| \left\|e_N\right\|_V \leq C \left( \left\|r_N\right\|\left\|e_N\right\|_V + \left\|e_N\right\|\left\|e_N\right\|_V\right).
\end{equation*}
The third remainder can be bounded using (\ref{eqn:Assumptions})
\begin{equation*}
	\rho_3 \leq C \left(\left\|u-u_N\right\| + \left\|Ju - Ju_N\right\|\right)\left\|e_N\right\| \leq C \left(\left\|r_N\right\| + \left\|e_N\right\|+ \left\|J(u - u_N)\right\|\right)\left\|e_N\right\|.
\end{equation*}
Now, the nonlocal operator $Ju$ is defined by (\ref{eqn:NonlocalOperator}) and hence
\begin{equation*}
	\begin{split}
		\left\|J(u - u_N)(t)\right\| 
		&\leq \int_0^\tau K(s) \left\|u(t-s)-u_N(t-s)\right\| ds \\
		&\leq \int_0^\tau K(s) \left\|r_N(t-s)\right\|ds + \int_0^\tau K(s) \left\|e_N(t-s)\right\|ds \\
		&\leq \int_0^t K(t-s) \left\|r_N(s)\right\|ds + \int_0^t K(t-s) \left\|e_N(s)\right\|ds,
	\end{split}
\end{equation*}
where in the last inequality we have changed the integration variable $s \rightarrowtail t - s$ and used the fact that the integral of a positive function over $[0,t]$ is larger than over $[t-\tau, t]$. Finally, we can combine all estimates of $\rho_i$ and use $\epsilon$-Cauchy inequality where necessary to extract $\left\|e_N\right\|_V^2$ and bound products of norms in terms of the sum of their squares. In effect, we arrive at
\begin{equation*}
	\begin{split}
		\frac{1}{2}\frac{d}{dt}\left\|e_N\right\|^2 &+ D_- \left\|e_N\right\|_V^2 
		\leq \frac{D_-}{2}\left\|e_N\right\|_V^2 \\
		&+ C\left(\left\|r_{N,t}\right\|^2 + \left\|e_N\right\|^2 + \int_0^t K(t-s) \left\|r_N(s)\right\|^2 ds + \int_0^t K(t-s) \left\|e_N(s)\right\|^2 ds \right),
	\end{split} 
\end{equation*}
whence
\begin{equation*}
	\frac{d}{dt}\left\|e_N\right\|^2 \leq C\left(\left\|r_{N,t}\right\|^2 + \int_0^t K(t-s) \left\|r_N(s)\right\|^2 ds + \left\|e_N\right\|^2 + \int_0^t K(t-s) \left\|e_N(s)\right\|^2 ds \right).
\end{equation*}
By integrating the above, we obtain
\begin{equation*}
	\begin{split}
		\left\|e_N(t)\right\|^2 &\leq \left\|e_N(0)\right\|^2 +  C\left(\int_0^t \left\|r_{N,t}(z)\right\|^2 dz + \int_0^t \left(\int_0^{t-s} K(s) ds\right)\left\|r_N(z)\right\|^2 dz \right.\\
		&\left.+ \int_0^t\left[1+\left(\int_0^{t-z} K(s) ds\right)\right]\left\|e_N(z)\right\|^2 dz \right).
	\end{split}
\end{equation*}
Since $K$ is integrable, its integral is continuous, hence bounded, and
\begin{equation*}
	\left\|e_N(t)\right\|^2 \leq \left\|e_N(0)\right\|^2 +  C\left(\int_0^t \left(\left\|r_{N,t}(z)\right\|^2+ \left\|r_N(z)\right\|^2\right) dz + \int_0^t \left\|e_N(z)\right\|^2 dz \right),
\end{equation*}
and we can invoke Gr\"onwall-Bellman's Lemma (Lemma \ref{lem:GB}) to arrive at
\begin{equation*}
	\begin{split}
		\left\|e_N(t)\right\|^2 
		&\leq \left(\left\|e_N(0)\right\|^2 +  C\int_0^t  \left(\left\|r_{N,t}(z)\right\|^2 + \left\|r_N(z)\right\|^2\right) dz\right) e^{C t} \\
		&\leq C \left(\left\|e_N(0)\right\|^2 + \int_0^t\left(\left\|r_{N,t}(z)\right\|^2 + \left\|r_N(z)\right\|^2\right) dz\right),
	\end{split}
\end{equation*}
since $t\in [0,t_0]$. This, Lemma \ref{lem:RitzProjection} and the fact that $u - u_N = r_n + e_N$ implies the assertion and completes the proof. 
\end{proof}

\section{Weighted linear time-stepping scheme}
We would like to fully discretize (\ref{eqn:GalerkinWeakForm}) to obtain a time-stepping numerical scheme. However, to reduce the computational cost, we would like to obtain a linear method. To this end, we will use the extrapolation of the nonlinear coefficients (see \cite{Tho07,Lus79}) along with a $\theta$-weighed scheme. 

As a preparation, we introduce the temporal grid on $[0,t_0]$ with a step $h$ is defined as
\begin{equation*}
t_n = n h, \quad h = \frac{t_0}{N_0},
\end{equation*}
where $N_0$ is the number of subintervals. Furthermore, if $U^n$ is a grid function, that is a function defined on $t_n$ with $n=0, ..., N_0$ that can be though as a piecewise constant function on $[0,t_0]$, we define the usual finite difference
\begin{equation*}
\delta U^n = \frac{U^{n+1}- U^n}{h}.
\end{equation*}
Furthermore, we have to discretize the nonlocal operator (\ref{eqn:NonlocalOperator}) and in general the discretization can be written for $\tau = h M$
\begin{equation}
\label{eqn:NonlocalOperatorDiscretization}
J_h U^n = \sum_{i=0}^{M} w_i(K) U^{n-i} + \rho_0(h),
\end{equation}
where $\rho_0(h)$ is the local consistency error, and $w_i(K)$ are weights. In particular, we can choose the rectangle rule to have for $0\leq i < M$
\begin{equation}
\label{eqn:NonlocalRectangle}
w_i(K) = 
\begin{cases}
	\displaystyle{\int_{t_{i}}^{t_{i+1}} K(s) ds}, & 0\leq i < M, \\
	0, & i = M,
\end{cases}
\end{equation}
or trapezoid rule
\begin{equation}
\label{eqn:NonlocalTrapezoid}
w_i(K) = 
\begin{cases}
	\displaystyle{\int_{0}^{h} K(s) \left(1-\frac{s}{h}\right)ds}, & i = 0, \vspace{4pt}\\
	\displaystyle{\int_{t_{i}}^{t_{i+1}} K(s) \left(1-\frac{s-t_i}{h}\right)ds + \int_{t_{i-1}}^{t_i} K(s) \, \frac{s-t_{i-1}}{h}ds}, & 1\leq i < M, \vspace{4pt}\\
	\displaystyle{\int_{\tau-h}^{\tau} K(s) \, \frac{s-\tau+h}{h}ds}, & i = M.
\end{cases}
\end{equation}
Provided sufficient smoothness, the orders of the above quadratures are: first, for the rectangle, and second for the trapezoid. Note that we have used the so-called product integration rule, that is, we have discretized the unknown function $U$ leaving the exact integral of the kernel. This procedure guarantees the maximal order of convergence for sufficiently smooth $U$ (see \cite{Lin85}). 

To state the method, we define the $\theta$-averaged value
\begin{equation*}
\widehat{U}^{n-\theta} := \theta U^{n-1} + (1-\theta) U^n, \quad 0\leq \theta\leq 1,
\end{equation*}
and its extrapolation through the past two time steps
\begin{equation}
\label{eqn:Extrapolation}
\overline{U}^{n-\theta} := (2-\theta) U^{n-1} - (1-\theta) U^{n-2}, \quad 0\leq \theta\leq 1.
\end{equation}
It can be easily seen by Taylor series expansion that
\begin{equation*}
|\widehat{U}^{n-\theta} - \overline{U}^{n-\theta}| = O(h^2), \quad h\rightarrow 0^+.
\end{equation*}
Now, we fix $N>0$ as the number of terms in the Galerkin semi-discrete solution of (\ref{eqn:GalerkinWeakForm}) and set $U^n$ to be the numerical approximation to $u_N(t_n)$. Note that we will omit writing subindex $N$ in our fully discrete approximation. We then propose the following scheme for solving (\ref{eqn:MainEq})
\begin{equation}
\label{eqn:NumericalScheme}
(\delta U^n, v) + a(D(\extrp); \widehat{U}^{n-\theta}, v) = (f_h(\extrp), v), \quad v\in V_N,
\end{equation} 
where $f_h$ is defined as
\begin{equation}
\label{eqn:fh}
f_h(U^n) = f(x, t_n, U^n, J_h U^n),
\end{equation}
which is a full discretization of the source nonlinearity. Note that since we are using extrapolation (\ref{eqn:Extrapolation}) in $D$ and $f$ we are required to solve only a linear system of equations in each time step. This technique of removing nonlinearity is a classical move developed in \cite{Lus79}. Due to the weighted nature of (\ref{eqn:NumericalScheme}) we obtain the first-order backward Euler method for $\theta = 1$ and second order Crank-Nicolson scheme for $\theta = 1/2$. Accordingly, with respect to the requirements, one can choose either the rectangle or a trapezoid quadrature in (\ref{eqn:NonlocalOperatorDiscretization}). 

To complete the numerical scheme, we have to state the initialization procedure. Since the extrapolation produced a two-step method, we have to carefully start the iteration for the order to be preserved. We use the predictor-corrector method. The predictor step is based on setting $U^0 = \psi_N(0)$ and solving for $W$ in
\begin{equation}
\label{eqn:Predictor}
\text{(P)} \quad h^{-1}(W-U^0, v) + a(D(U^0); U^{1-\theta}, v) = (f_h(U^0), v) \quad v \in V_N.
\end{equation}
Then, we correct the value of $W$ by 
\begin{equation}
\label{eqn:Corrector}
\text{(C)} \quad h^{-1}(U^1-W, v) + a(D(W^{1-\theta}); U^{1-\theta}, v) = (f_h(W^{1-\theta}), v) \quad v \in V_N,
\end{equation}
where $W^{1-\theta} = (1-\theta) U^0 + \theta W$. This gives us two starting values $U^0$, $U^1$ that can be plugged into the time-stepping (\ref{eqn:NumericalScheme}).

We now move to the convergence result.
\begin{thm}[Convergence of the full discrete scheme]
Let $u(t)\in H^{2m}(0,1)$ for each $t\in[0,t_0]$ with $m\geq 1$. Furthermore, assume that $u_x$, $u_t$, and $u_{tt}$ are bounded. Then, 
\begin{equation*}
	\left\|u^n - U^n \right\| \leq C \left(N^{-2m} + \rho_0(h) \left(\theta-\frac{1}{2}\right)h + h^2\right),
\end{equation*}
where the constant $C$ depends on $u$, $\psi$ and all their derivatives, and $\rho_0(h)$ is defined in (\ref{eqn:NonlocalOperatorDiscretization}). 
\end{thm}
\begin{proof}
Similarly as in the proof of the semidiscrete scheme, we start by decomposing the error
\begin{equation*}
	u^n - U^n = u^n - R_N u^n + R_N u^n - U^n =: r^n + e^n, 
\end{equation*}
where $u^n = u(t_n)$ and we keep our convention not to write $N$ explicitly (since it is fixed). Then, we start writing the error equation with $v\in V_N$
\begin{equation*}
	\begin{split}
		(\delta e^n, v) &+ a(D(\extrp); \hat{e}^{n-\theta}, v) \\
		&= (\delta R_N u^n, v) + a(D(\extrp); R_N \widehat{U}^{n-\theta}, v)- (\delta U^n, v) - a(D(\extrp); \widehat{U}^{n-\theta}, v)\\
		& = (\delta R_N u^n, v) + a(D(\extrp); R_N \widehat{U}^{n-\theta}, v) - (f_h(\extrp), v),
	\end{split}
\end{equation*}
where in the last equality we have used (\ref{eqn:NumericalScheme}). Expanding further, we have
\begin{equation*}
	\label{eqn:ErrorEquation}
	\begin{split}
		(\delta e^n, v) &+ a(D(\extrp); e^{n-\theta}, v) \\
		& = (\delta R_N u^n, v) + a(D(\extrp)-D(u^{n-\theta}); R_N \widehat{U}^{n-\theta}, v) \\
		&+ a(D(u^{n-\theta}); R_N \widehat{U}^{n-\theta} - R_N u^{n-\theta}, v) + a(D(u^{n-\theta}); R_N u^{n-\theta}, v) - (f_h(\extrp), v) \\
		&= (\delta R_N u^n - u_t^{n-\theta}, v) + a(D(\extrp)-D(u^{n-\theta}); R_N \widehat{U}^{n-\theta}, v) \\
		&+ a(D(u^{n-\theta}); R_N \widehat{U}^{n-\theta} - R_N u^{n-\theta}, v) + (f(u^{n-\theta}) - f_h(\extrp), v)
	\end{split}	
\end{equation*}
where this time we have used the main equation (\ref{eqn:WeakForm}). In the above, we would like to put $v = \hat{e}^{n-\theta}$ in order to derive estimates on the error. Before that, however, note that
\begin{equation*}
	(\delta e^n, \hat{e}^{n-\theta}) = \frac{1}{2}\delta\left\|e^n\right\|^2 + h \left(\theta-\frac{1}{2}\right) \left\|\delta e^n\right\|^2,
\end{equation*}
which can be shown by expanding the definitions of $\delta e^n$ and $\hat{e}^{n-\theta}$. Note that we see that taking $\theta = 1/2$ kills the $O(h)$ term. Therefore, with the aforementioned choice of $v$, we can write
\begin{equation*}
	\frac{1}{2}\delta\left\|e^n\right\|^2 + h \left(\theta-\frac{1}{2}\right) \left\|\delta e^n\right\|^2 + D_- \left\|\hat{e}^n\right\|^2_V \leq \rho_1 + \rho_2 + \rho_3 + \rho_4,
\end{equation*}
where the remainders $\rho_i$ are understood from (\ref{eqn:ErrorEquation}). We will now bound each of them.

We start with the difference in time derivatives
\begin{equation*}
	\left\|\delta R_N u^n - u_t^{n-\theta}\right\| \leq \left\|\delta r^n \right\| + \left\| \delta u^n - u_t^{n-\theta} \right\|.
\end{equation*}
Now, by Taylor expansion at $t = t_{n-\theta} = (n-\theta)h$ we obtain
\begin{equation*}
	\begin{split}
		\delta u^n - u_t^{n-\theta} &= h^{-1} \left(u^{n-\theta} + h\theta u^{n-\theta}_t - u^{n-\theta} + (1-\theta) h u^{n-\theta}_t + \right. \\
		&\left. \frac{1}{2}h^2\theta^2 u^{n-\theta}_{tt}-h^2(1-\theta)^2 u^{n-\theta}_{tt} + O(h^3)\right) - u^{n-\theta}_t \\
		&= h\left(\theta-\frac{1}{2}\right)u^{n-\theta}_{tt} + O(h^2), \quad h\rightarrow 0^+.
	\end{split}
\end{equation*}
Whence, by Lemma \ref{lem:RitzProjection} we have
\begin{equation*}
	\rho_1 \leq C \left(N^{-2m} + h\left(\theta-\frac{1}{2}\right) + h^2\right) \left\|\hat{e}^{n-\theta}\right\|.
\end{equation*}
Next, the remainder associated with $a$ can be bounded thanks to (\ref{eqn:Assumptions})
\begin{equation*}
	\rho_2 \leq C \left\|R_N \widehat{u}^{n-\theta}\right\|_\infty \|\overline{U}^{n-\theta} - u^{n-\theta}\| \left\|\hat{e}^{n-\theta}\right\|_V,
\end{equation*}
and the difference between extrapolated and the value of the exact solution can be estimated as follows
\begin{equation*}
	\begin{split}
		\|\overline{U}^{n-\theta} - u^{n-\theta}\| &\leq \|\overline{U}^{n-\theta} - \overline{u}^{n-\theta}\| + \|\overline{u}^{n-\theta} - u^{n-\theta}\| \leq \left\|\overline{e}^{n-\theta}\right\|+\left\|\overline{r}^{n-\theta}\right\| + \|\overline{u}^{n-\theta} - u^{n-\theta}\| \\
		&\leq C \left(\left\|e^{n-1}\right\| + \left\|e^{n-2} \right\| + N^{-2m} + h^2\right),
	\end{split}
\end{equation*}
since, by construction, $\overline{u}^{n-\theta}$ approximates $u^{n-\theta}$ to the second order. Next, we proceed with the third remainder to obtain
\begin{equation*}
	\rho_3 \leq C\left\|R_N \widehat{u}^{n-\theta} - R_N u^{n-\theta} \right\|_V \left\| \hat{e}^{n-\theta} \right\|_V \leq D_-\left\| \hat{e}^{n-\theta} \right\|_V^2 + C\left\|R_N \widehat{u}^{n-\theta} - R_N u^{n-\theta} \right\|_V^2,
\end{equation*}
where we have used the $\epsilon$-Cauchy inequality to extract the $V$-norm of the error. However, exactly as above, due to Taylor expansion at $t = t_{n-\theta} = (n-\theta)h$ we have for any function $y$ of time
\begin{equation*}
	\left\|\widehat{y}^{n-\theta} - y^{n-\theta}\right\| \leq  \frac{1}{2}\theta(1-\theta)h^2 \widetilde{y}_{tt}^{n-\theta}.
\end{equation*}
To apply this estimate to the $\rho_3$ we have to show that the gradient of the second time derivative of Ritz projection is bounded in $V$. To this end, differentiate the definition (\ref{eqn:RitzProjection}) twice to obtain
\begin{equation*}
	a(D(u); R_N u - u, v) = -a(D(u)_{tt}; r, v) - 2 a(D(u)_t; r_t,v) + a(D(u); u_{tt},v).
\end{equation*}
From here, with a choice $v = (R_N u)_{tt}$ we have
\begin{equation*}
	\left\|(R_N u)_{tt}\right\|_V \leq C \left(\left\|r\right\|_V + \left\|r_t\right\|_V + \left\|u_{tt}\right\|_V\right) \leq C\left\|u_{tt}\right\|_V, 
\end{equation*}
for sufficiently large $N$ by Lemma \ref{lem:RitzProjection}. Whence, according to Lemma \ref{lem:RitzProjection} we finally obtain 
\begin{equation*}
	\rho_3 \leq D_-\left\| \hat{e}^{n-\theta} \right\|_V^2 + C \left(h \left(\theta-\frac{1}{2}\right) + h^2\right)^2. 
\end{equation*}
The last remainder involves a nonlocal operator. First, we estimate the difference by the regularity assumption on $f$
\begin{equation*}
	\left\|f(u^{n-\theta}) - f_h(\extrp)\right\| = C \left(\left\|\extrp - u^{n-\theta}\right\| + \left\|J_h\extrp - J u^{n-\theta}\right\| \right).
\end{equation*}
Now, by (\ref{eqn:NonlocalOperatorDiscretization}) we have
\begin{equation*}
	\begin{split}
		\left\|J_h\extrp - J u^{n-\theta}\right\| 
		&\leq C \rho_0(h) + \sum_{i=0}^M w_i(K) \left\|\overline{U}^{n-\theta-i} - u^{n-\theta-i}\right\| \\
		&\leq C \rho_0(h) + \sum_{i=0}^M w_i(K) \left( \left\|\overline{r}^{n-\theta}\right\| + h^2\right) + \sum_{i=0}^M w_i(K) \left\|\overline{e}^{n-\theta-i}\right\| \\
		&\leq C (\rho_0(h) + h^2 + N^{-2m})+ \sum_{i=0}^M w_i(K) \left\|\overline{e}^{n-\theta-i}\right\|,
	\end{split}
\end{equation*}
by integrability of the kernel $K$. Hence,
\begin{equation*}
	\rho_4 \leq C \left(\rho_0(h) + h^2 + N^{-2m} + \sum_{i=0}^M w_i(K) \left\|\overline{e}^{n-\theta-i}\right\|\right) \left\|\hat{e}^{n-\theta}\right\|. 
\end{equation*}
Putting all estimates of $\rho_i$ that we have obtained so far and defining
\begin{equation*}
	\rho(N,h) := \rho_0(h) + h^2 + N^{-2m},
\end{equation*} 
leads to
\begin{equation*}
	\frac{1}{2} \delta \left\|e^n\right\|^2 \leq C \left(\rho(N,h)^2 + \left\|e^{n-1}\right\|^2 + \left\|e^{n-2}\right\|^2 + \left(\sum_{i=0}^M w_i(K) \left\|\overline{e}^{n-\theta-i}\right\|\right)^2\right),
\end{equation*}
where we have again used the $\epsilon$-Cauchy inequality to extract the $L^2$ norm of the errors. Now, since $K$ is integrable, the weights $w_i(K)$ are bounded which along with a simple real number inequality $(a+b)^2 \leq 2(a^2+b^2)$ yields a nonlocal recurrence
\begin{equation*}
	\delta \left\|e^n\right\|^2 \leq C \left(\rho(N,h)^2 + \left\|e^{n-1}\right\|^2 + \left\|e^{n-2}\right\|^2 + \sum_{i=0}^M \left(\left\|e^{n-1-i}\right\|^2 + \left\|e^{n-i-2}\right\|^2 \right)\right),
\end{equation*} 
which, by changing the summation variable and enlarging the constant $C$,  can be transformed into 
\begin{equation*}
	\left\|e^n\right\|^2 \leq (1+Ch)\left\|e^{n-1}\right\|^2 + Ch \sum_{i=2}^{M+1} \left\|e^{n-i}\right\|^2 + Ch \rho(N,h)^2.
\end{equation*}
Now, adding terms on both sides we can write
\begin{equation*}
	\left\|e^n\right\|^2 + C h \sum_{i=1}^{M} \left\|e^{n-i}\right\|^2 \leq (1+Ch) \left(\left\|e^{n-1}\right\|^2 + Ch \sum_{i=2}^{M+1} \left\|e^{n-i}\right\|^2\right) + Ch \rho(N,h)^2.
\end{equation*}
Set
\begin{equation*}
	c_n := \left\|e^n\right\|^2 + C h \sum_{i=1}^{M} \left\|e^{n-i}\right\|^2,
\end{equation*}
to obtain a simple recurrence
\begin{equation*}
	c_n \leq (1+Ch) c_{n-1} + Ch \rho(N,h)^2.
\end{equation*}
By iteration, we then have
\begin{equation*}
	\begin{split}
		c_n &\leq (1+Ch)^{n-1}c_1 + Ch\rho(N,h)^2 \sum_{i=0}^{n} (1+Ch)^i \\
		&\leq (1+Ch)^{n-1} c_1 + (1+Ch)^{n+1} \rho(N,h)^2 \leq e^{Ct_0} \left(c_1 + \rho(N,h)^2\right),
	\end{split}
\end{equation*}
and therefore,
\begin{equation*}
	\left\|e^n\right\|^2 \leq C \left(\left\|e^1\right\|^2 + C h \sum_{i=1}^{M} \left\|e^{1-i}\right\|^2 + \rho(N,h)^2\right).
\end{equation*}
Now, we are left with estimating the error made in the first predictor-corrector step (\ref{eqn:Predictor})-(\ref{eqn:Corrector}) since the sum above involves the initial condition. To this end, set $g^1 = W - R_N u^1$. Similarly as above, we can show that for the predictor (\ref{eqn:Predictor}) we have
\begin{equation*}
	\delta\left\|g^1\right\|^2 \leq C \left(\left\|U^0-u^{1-\theta}\right\|^2 + \rho(N,h)^2\right).
\end{equation*}
Then,
\begin{equation*}
	\left\|U^0-u^{1-\theta}\right\|^2 \leq \left\|e^0\right\|^2 + \left\|r^0\right\|^2 + Ch \leq C \left(\left\|e^0\right\|^2+ N^{-2m} + h\right),
\end{equation*}
and it follows that
\begin{equation*}
	\left\|g^1\right\|^2 \leq C \left(\left\|e^0\right\|^2+ h\left(N^{-2m} + h\right)^2 + h\rho(N,h)^2\right).
\end{equation*}
Next, we move to the corrector stage (\ref{eqn:Corrector}) to obtain
\begin{equation*}
	\delta \left\|e^1\right\|^2 \leq C\left(\|\widehat{W}^{1-\theta} -u^{1-\theta}\|^2 + \rho(N,h)^2\right).
\end{equation*}
Reasoning as above, we have
\begin{equation*}
	\begin{split}
		\|\widehat{W}^{1-\theta} -u^{1-\theta}\| 
		&\leq \left\|\widehat{g}^{1-\theta}\right\| + \left\|\widehat{R_N} u^{1-\theta} - u^{1-\theta} \right\| \leq \left\|e^0\right\| + \left\|g^1\right\| + \left\|\widehat{R_N} u^{1-\theta} - u^{1-\theta} \right\| \\
		& \leq \left\|e^0\right\| + \left\|g^1\right\| + C N^{-2m} + \left\|\widehat{R_N} u^{1-\theta} - R_N u^{1-\theta} \right\|.
	\end{split}
\end{equation*}
Using the estimate on $\|g^1\|$ from the prediction stage, we obtain
\begin{equation*}
	\begin{split}
		\|\widehat{W}^{1-\theta} -u^{1-\theta}\| 
		&\leq C \left(\|e^0\| + h^\frac{1}{2} (N^{-2m}+ h) + h^\frac{1}{2} \rho(N,h)+N^{-2m} + \rho(N,h)\right) \\
		&\leq C\left(\|e^0\| + h^\frac{3}{2} + \rho(N,h)\right).
	\end{split}
\end{equation*}
Now, by going back to the estimate on the finite difference of the error, we finally have
\begin{equation*}
	\left\|e^1\right\| \leq C \left(\|e^0\| + h^2 + \rho(N,h)\right).
\end{equation*}
The estimate of the initial value errors $\left\|e^{-i}\right\|$ where $i = 0,1,2,..., M$ follows from
\begin{equation*}
	\left\|e^{-i}\right\| = \left\|R_N u^{-i} - P_N \psi^{-i}\right\| \leq C N^{-2m},
\end{equation*}
what ends the proof.
\end{proof}

\section{Implementation and numerical illustration}
Now we are concerned about the practical use of the aforementioned algorithm and its efficient implementation. Let $U^n$ be expanded into our basis (\ref{eqn:Basis})
\begin{equation*}
U^n = \sum_{i=0}^N y_i^n \phi_i,
\end{equation*}
then, plugging into (\ref{eqn:NumericalScheme}), by orthonormality of $\left\{\phi_i\right\}_i$ we obtain
\begin{equation}
\label{eqn:NumericalSchemeImplementation}
\left(I + (1-\theta)h A(\overline{U}^{n-\theta})\right)\textbf{y}^n = \left(I - \theta h A(\overline{U}^{n-\theta})\right)\textbf{y}^{n-1} + h \, \textbf{f}_h(\overline{U}^{n-\theta}), \quad n\geq 2,
\end{equation}
where $\textbf{y}^n = \left\{y^n_i\right\}_{i=0}^N$ is a vector of solutions, while the stiffness matrix $A = \left\{A_{ij}\right\}_{i,j=0}^N$ and the load vector $\textbf{f}_h = \left\{f_{h,i}\right\}_i$ are defined by
\begin{equation*}
A_{ij} = (D(\extrp); \phi_i, \phi_j), \quad f_{h,i} = (f_h(\extrp), \phi_i).
\end{equation*} 
In the linear case, the stiffness matrix is diagonal $A = \text{diag}(\lambda_i)_i$ with eigenvalues (\ref{eqn:Eigenfunctions}). The implementation requires solving (\ref{eqn:NumericalSchemeImplementation}) in each time step for $\textbf{y}^n$ which reduces to inverting the nonsingular matrix $I + (1-\theta)h A(\overline{U}^{n-\theta})$. For the linear diffusion, this matrix is constant over time and the inversion needs to be done only at the initialization phase. 

We would like to numerically verify the above results concerning convergence. To this end, we will calculate the order of the method for the Crank-Nicolson scheme with $\theta = 1/2$, with trapezoidal quadrature (\ref{eqn:NonlocalTrapezoid}) and two choices of nonlocal operator kernels. The first one is the Gaussian as was suggested in the original work \cite{Bha82}
\begin{equation*}
K(s) = G(s) := A e^{-\frac{(\tau-2s)^2}{8\sigma^2}},
\end{equation*}
where $A$ is the amplitude and $\sigma^2$ is the variance. This kernel is responsible for a short memory effect due to the exponential decay of its tail. In contrast to that, we can also consider a long memory kernel obtained by the power law describing the heavy tail
\begin{equation}
\label{eqn:NonlocalOperatorFrac}
K(s) = K_\alpha(s) = \frac{1}{\Gamma(\alpha)} s^{\alpha-1}, \quad \alpha > 0.
\end{equation}
In the above, the prefactor involving the gamma function has been chosen to be consistent with the Liouville fractional integral
\begin{equation*}
I^\alpha y(t) = \frac{1}{\Gamma(\alpha)} \int_{-\infty}^t (t-s)^{\alpha-1} y(s) ds,
\end{equation*}
which arise after substitution $s \rightarrow t-s$ and allowing for the infinite memory with $\tau\rightarrow\infty$. Fractional integrals and fractional derivatives are important in many applications and there is a very vigorous research going on this topic: both applied and pure. The interested reader can find additional information in \cite{Met00}. 

In our numerical examples, we would like to test two cases: either linear diffusion with nonlocal operator or nonlinear diffusion without nonlocality. Errors are calculated in the $L^2$ norm and the coding is done in Julia programming language. It is a free high-performance and high-level dynamic programming language that, apart from many others uses, performs very well in numerical simulations where one has to conduct a large scale computations. Julia neatly combines ease of coding with speed of execution (for an introduction for scientists see \cite{perkel2019julia}). We have implemented several performance mechanisms:
\begin{itemize}
\item All $x$-integrals are calculated using Gaussian quadrature with pregenerated Legendre weights.
\item Where possible, we utilize parallel computing with several threads/cores (multi-threading). For example, the stiffness matrix and load vector entries, as well as quadrature weights (\ref{eqn:NonlocalRectangle})-(\ref{eqn:NonlocalTrapezoid}) can readily be calculated effectively in this way.
\item Quadrature weights need only be generated once as soon as the kernel $K$ and the time step $h$ are fixed. 
\item For linear diffusivity, the stiffness matrix is pregenerated. For nonlinear diffusion, we utilize multi-threading. 
\end{itemize}   
We have found that the approach based on parallelism is highly efficient with respect to the single core calculations. 

Below, for concreteness we set $\tau = 0.4$ and $t_0 = 0.5$. Moreover, the initial condition is always taken to be
	\begin{equation*}
		\psi(x,s) = \frac{\cos(\pi x)}{1+s}, 
	\end{equation*}
	which models a high temperature at the equator and low at the pole. Overall, we solve three cases of our problem
	\begin{equation}
		\label{eqn:Cases}
		\begin{cases}
			D(u) = e^{-\beta u}, \\
			f(x,t,u,w) = u(1-u^2), \\
		\end{cases}
		\quad
		\begin{cases}
			D(u) = 1, \\
			f(x,t,u,w) = w(1-w), \\
			K(s) = G(s), \\
		\end{cases}
		\quad
		\begin{cases}
			D(u) = 1, \\
			f(x,t,u,w) = w(1-w), \\
			K(s) = K_\alpha(s), \\
		\end{cases}
	\end{equation} 
which test the effects of nonlinearity and nonlocality in several ways. The first example introduces pure nonlinear effects into the equation: both in the diffusivity and the source. The former has been chosen to satisfy positive definiteness while the latter to ensure three nontrivial steady states that are usually present in energy balance models. They describe (stable) cold, (unstable) intermediate, and (stable) warm climates (more information can be found for ex. in \cite{Fow11}). Next, two examples investigate the notion of history: short (Gaussian) and long memory (fractional integral) kernels. For the latter, we choose a representative value $\alpha = 0.5$ since the results for others do not present any significant qualitative differences.

First, to check the spectral accuracy of the spatial discretization, we fix the time step to be $h = 2\times 10^{-2}$ and compare solutions obtained for different $N$ with the one calculated for $N_{max} = 30$. The latter is treated as "exact". The time step $h$ can also be chosen smaller and in that case the error would decay even faster. The comparison is done at $t = 0.125$ to not let diffusive effects to force the solution to decay, which can happen for larger times. Results of our simulations are shown in Fig. \ref{fig:rzadX} where a semi-log plot of errors in presented. As we can see, the spectral accuracy is confirmed for all considered cases, that is the logarithm of the error behaves approximately linearly. The error saturates at machine epsilon (that is $\epsilon_M = 2.2\times 10^{-16}$) for $N \approx 11-13$ for all problems which means that using several degrees of freedom can grant the (numerically) perfect accuracy. We can also see that the convergence to zero is slightly faster for no or short memory examples. The error could be made even smaller if we decreased the time step, however, we wanted to clearly show the error trend. 

\begin{figure}
\centering
\includegraphics[scale = 0.8]{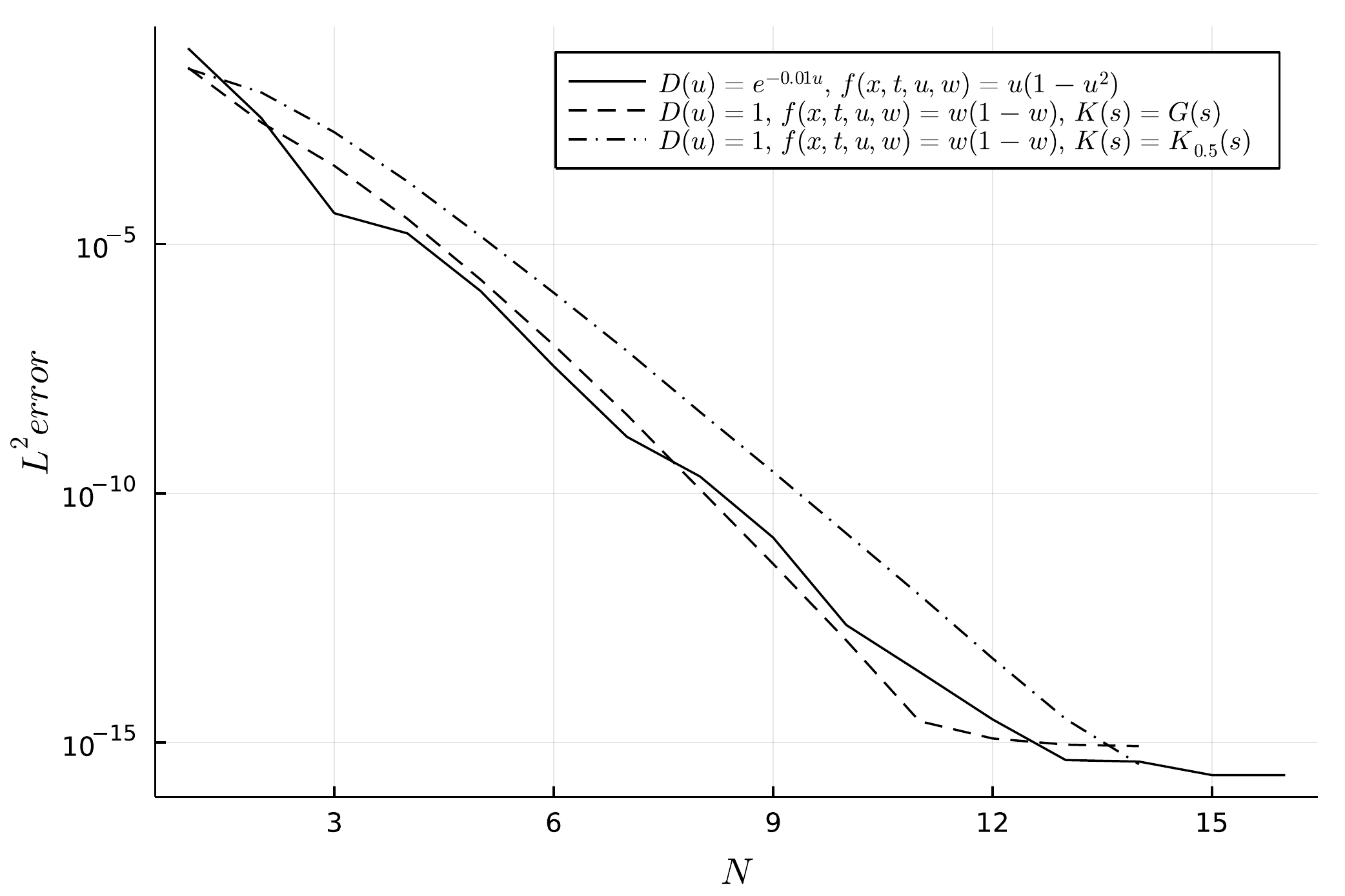}
\caption{Numerically calculated $L^2$ error between solutions with $N_{max} = 30$ for different $N$ and problems given in (\ref{eqn:Cases}). }
\label{fig:rzadX}
\end{figure}

The temporal order is calculated in a similar way. Here, we use $N = N_{max} = 20$. As we have seen in above computations, this number of degrees of freedom is sufficient to grant a complete spatial accuracy. Our ,,exact'' solution is then precalculated for $h_{min} = 2\times 10^{-3}$ and compared with calculations done with larger steps. The time of comparison is taken to be $t_0$. In Fig. \ref{fig:rzadT} we present the log-log plot of the $L^2$ error for different values of $h$. As we can see, for all considered cases the order of convergence is confirmed to be equal to $2$, that is the lines are parallel to $h^2$ for large values of $h^{-1}$ (small $h$). Note that the weakly singular kernel in the fractional integral (\ref{eqn:NonlocalOperatorFrac}) does not impair the accuracy. This is because we have chosen to integrate \emph{exactly} the kernel using the product integration rule. The only possible slight loss of accuracy is seen in the nonlinear diffusivity example. However, the graph still confirms the method's second order. 

\begin{figure}
\centering
\includegraphics[scale = 0.8]{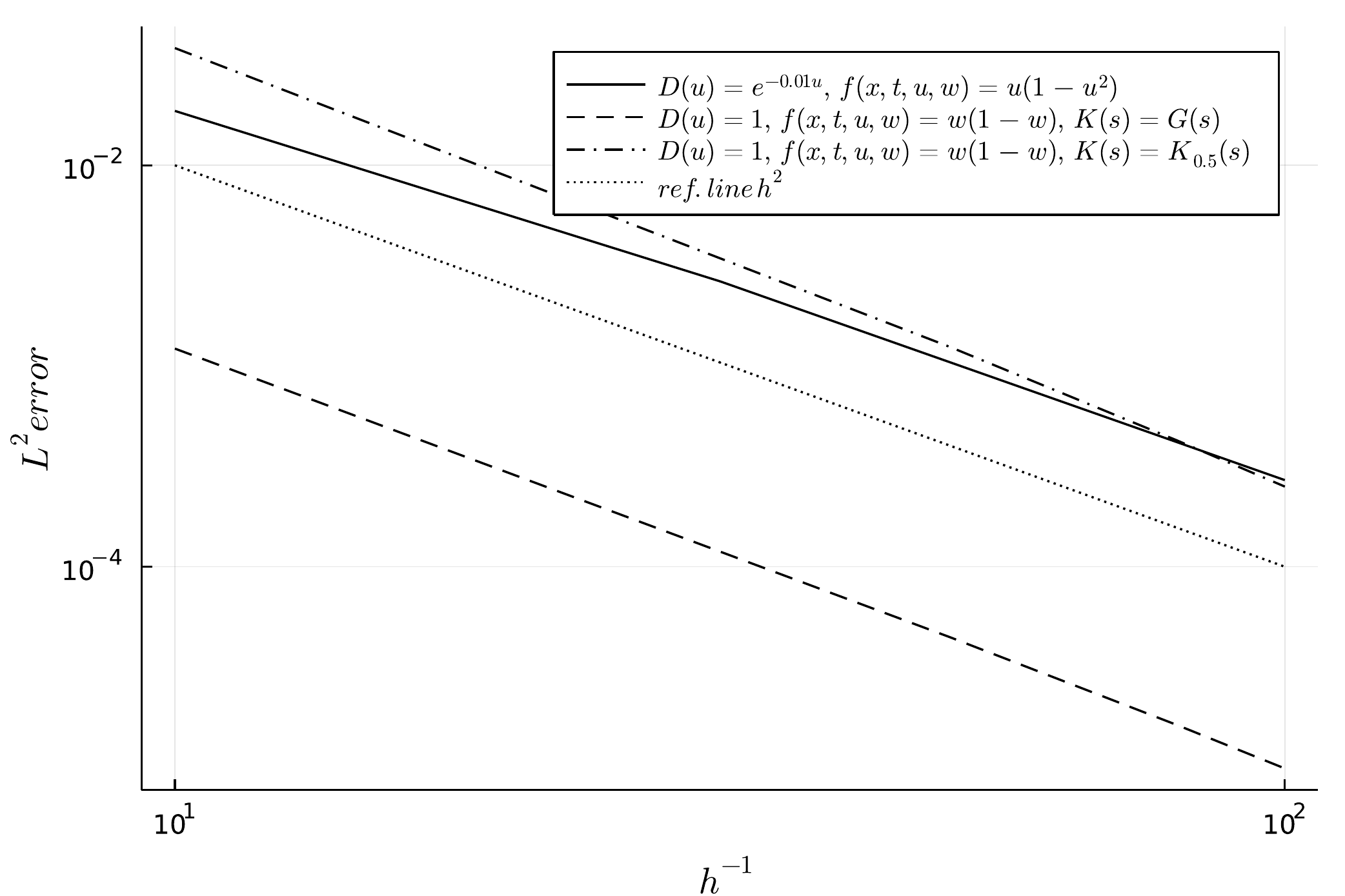}
\caption{Numerically calculated $L^2$ error with $h_{min} = 2\times 10^{-3}$ for different $h$ and problems given in (\ref{eqn:Cases}.}
\label{fig:rzadT}
\end{figure}

Lastly we evaluate the relative speed of our numerical method. This benchmark is based on comparison with the standard Finite Element Method with second-order interpolation in space and the method of lines adaptive implicit Runge-Kutta time integration. To provide the state-of-the-art computations, we have used the \textit{Wolfram Mathematica} symbolic computation environment. The specific numerical routine was then \textit{NDSolve} with appropriate options. To provide some objective, i.e., CPU independent data, we compute the ratio of the computation time for the FEM method and our spectral scheme required to obtain a fixed level of absolute error. That is, for a given tolerance, we calculate the following quantity
	\begin{equation}
		\label{eqn:CPUTime}
		\text{ratio} = \frac{\text{time for FEM}}{\text{time for spectral method}}
	\end{equation}		
	by choosing the minimal spatial degrees of freedom to obtain a given error. The problem that we solve is the following
	\begin{equation*}
		\begin{cases}
			D(u) = 1+u, \\
			f(x,t,u,w) = u(1+u), \\
		\end{cases}
		\quad u(x,0) = \cos(\pi x).
	\end{equation*}
Our results are presented in Tab. \ref{tab:CPUTime}. The superiority of using the spectral method is evident. As can be seen, it is $40-100$ times faster than the FEM of choice. This advantage comes from the fact that spectral methods are best suited for finding smooth solutions. Note that, if we were to expect less regular solutions, the FEM would be the method of choice. 

\begin{table}
\centering
\begin{tabular}{cccccccccc}
	\toprule
	absolute error level & $10^{-1}$ & $10^{-2}$ & $10^{-3}$ & $10^{-4}$ & $10^{-5}$ & $10^{-6}$ & $10^{-7}$ & $10^{-8}$ & $10^{-9}$ \\
	\midrule
	ratio & $133$ & $87$ & $85$ & $96$ & $112$ & $71$ & $53$ & $51$ & $40$ \\
	\bottomrule
\end{tabular}
\caption{The ratio (\ref{eqn:CPUTime}) of computation time for a given error level for two considered methods.}
\label{tab:CPUTime}
\end{table}

\section{Conclusion}
Motivated by the efficient use of the spectral method for solving equations of energy balance models, we have provided convergence proofs for Galerkin-Legendre scheme. Spectral approximation to the solution of the diffusive energy balance model has been known since its early days. However, throughout many years despite the successful use of the Legendre method, there has not been any rigorous analysis of convergence and stability. The model brings out a difficulty originating in degenerate diffusivity. Our treatment introduced some special weighted Sobolev space that made the convergence analysis feasible along with optimal estimates on the error. The use of spectral method is also beneficial when it comes to the treatment of nonlocal (memory) operators. Their numerical evaluation requires substantial computing power due to the history of the process. Since spectral methods have exponential accuracy it is then possible to reduce the number of degrees of freedom on the spatial part of the problem and use the remaining computational power to treat the memory effects on the temporal side.

Discretization in time using extrapolated coefficients grants a linear second-order method that can quickly compute the solution to the desired accuracy. Moreover, using a parallelized code fragment helped to improve the performance even further. Due to the nonlocal operator that requires to take into account the history of the evolution, the initial computational cost of the simulations can be high. Thanks to the spectral accuracy and multithread computations, we were able to reduce it. In future work, we will extend our method to some more general equations with nonlocal specific heat and diffusivity proportional to the gradient. That is to say, we plan to consider the following problem suggested, for example, in \cite{Dia97}
	\begin{equation*}
		c(x,T, JT) T_t = (d(T)(1-x^2)|T_x|^{p-2}T_x)_x + g(x, t, T, JT), \quad 1 < p < \infty, \quad x\in(0,1),
	\end{equation*}
with Neumann boundary condition and the usual choice of the initial condition and the memory operator defined in (\ref{eqn:NonlocalOperator}). Two difficulties arise here: one due to nonlinear and nonlocal heat capacity, and the other due to doubly degenerate diffusivity. The latter means that apart from degeneracy at $x = 1$ we also have to carefully deal with the situation when $T_x \approx 0$. Both of these generalizations may have a significant impact on the numerical analysis of the problem. 

It will also be both relevant and interesting to consider the free-boundary problem frequently found in conceptual climate models \cite{Dia20}. As we have mentioned in Section 2, following Budyko, one would like to introduce the concept of the ice line $x_i(t)$ separating the regions neighbouring with $T = T_i$, where $T_i$ is a prescribed temperature. The archetypal albedo could then be (\ref{eqn:AlbedoFreeBoundary}). An additional Stefan-like condition would then be the continuity of the flux at the free boundary $x = x_i(t)$. In that case, our spectral scheme could then be coupled with numerical approaches to solving free-boundary problems. Investigating both of these situations will contribute to our programme of a rigorous analysis of climatically relevant nonlinear and nonlocal models. 

\section*{Acknowledgement}
Ł.P. has been supported by the National Science Centre, Poland (NCN) under the grant Sonata Bis with a number NCN 2020/38/E/ST1/00153.

\bibliography{biblio2}
\bibliographystyle{plain}

\end{document}